\documentclass[11pt]{amsart}

\usepackage{amsmath}
\usepackage{amssymb}
\usepackage{graphicx}

\newtheorem{theorem}{Theorem}[section]

\newtheorem{lemma}[theorem]{Lemma}
\newtheorem{proposition}[theorem]{Proposition}
\theoremstyle{definition}
\newtheorem{definition}[theorem]{Definition}
\newtheorem{example}[theorem]{Example}
\newtheorem{remark}[theorem]{Remark}

\numberwithin{equation}{section}

\newcommand{\B}{\mathbb B}
\newcommand{\R}{\mathbb R}
\newcommand{\N}{\mathbb N}

\newcommand{\Sph}{\mathbb S}

\newcommand{\X}{\mathbb X}
\newcommand{\Y}{\mathbb Y}

\newcommand{\nullv}{{\bf 0}}
\newcommand{\inte}{{\rm int}\,}
\newcommand{\fron}{{\rm bd}\,}
\newcommand{\dom}{{\rm dom}\,}

\newcommand{\grph}[1]{{\rm grph}(#1)}
\newcommand{\epi}[1]{{\rm epi}(#1)}
\newcommand{\Inv}[1]{{#1}^{-\sharp}}
\newcommand{\SFix}[1]{{\rm SFix}(#1)}
\newcommand{\por}[1]{\le_{{}_{#1}}}

\newcommand{\ball}[2]{{\rm B}(#1, #2)}
\newcommand{\balll}[2]{{\rm B}\left(#1, #2\right)}
\newcommand{\dist}[2]{{\rm dist}\left(#1,#2\right)}
\newcommand{\exc}[2]{{\rm exc}\left(#1,#2\right)}
\newcommand{\excf}[2]{{\rm exc}_{#1,#2}}
\newcommand{\Haus}[2]{{\rm Haus}(#1, #2)}
\newcommand{\Inc}[2]{{\rm Inc}(#1, #2)}

\title[On a strong covering property of multivalued mappings]{On a strong covering property of multivalued mappings}

\author[A. Uderzo]{Amos Uderzo}

\address[A. Uderzo]{University of Milano-Bicocca, Dept. of
Mathematics and Applications, Milan, Italy}
\email{{\tt amos.uderzo@unimib.it}}



\begin{document}

\begin{abstract}
In this paper, a strong variant for multivalued mappings of the
well-known property of openness at a linear rate  is studied.
Among other examples, a simply characterized class of closed
convex processes between Banach spaces, which satisfies
such a covering behaviour, is singled out.
Equivalent reformulations of this property and its stability under
Lipschitz perturbations are investigated in a metric space setting.
Applications to the solvability of set-valued inclusions and to the
exact penalization of optimization problems with set-inclusion
constraints are discussed.
\end{abstract}

\maketitle

\begin{flushright}
to the memory of Aleksander Moiseevich Rubinov (1940-2006)
\end{flushright}

\vskip1cm


\section{Introduction and preliminaries}

The property of being surjective describes an elementary
set-theoretic behaviour of mappings which, in synergy with
specific (topological, metric, linear, and so on)
structures on the domain and
on the range sets, may afford valuable consequences.
A paradigma of this phenomenon can be seen in the celebrated
Banach-Schauder open mapping theorem for linear
operators: in a context, where the metric completeness
interacts with linearity and continuity, the property of a mapping
of being onto turns out to imply its openness (images of open sets remain open).
Even though some quantitative estimates did appear already in the
original statement of this result (see, for instance, Theorem 10
in Ch. X of \cite{Bana32}),
the merely topological formulation of it left somehow hidden certain
metric aspects of this ``openness preservation law". Nonetheless, the potential
of them was understood some years later, when the open
mapping theorem was extended in a local form to nonlinear
mappings by L.A. Lyusternik (see \cite{Lyus34}) and L.M. Graves
(see \cite{Grav50}). Their far-reaching extensions paved the
way to enlightening the interconnections of the quantitative surjective
behaviour of a mapping with the Lipschitzian behaviour of its
inverse and with error bounds for the solution set to the related
generalized equations (see \cite{Mord93,Peno89}), and, consequently,
to explore links with the metric fixed point theory (see \cite{Arut07,DonRoc14,Ioff14}).
This led to distil the notion of metric regularity as an essential
tool for the analysis of various stability and sensitivity issues
in modern variational analysis, optimization and control theory (see historical
comments in \cite{DonRoc14,Ioff00,Mord06,RocWet98}).
As a consequence, the covering behaviour of single as well
as of set-valued mappings has been the main subject of many
investigations (see, among the others, \cite{Arut07,BorZhu88,
DmMiOs80,Fran87,Mord93,MorSha95,YeYaKi98}). In them, depending on possible
applications or on specific issues of the related theory to be
investigated, several variations of the concept of covering
behaviour itself have been considered: for instance, 
a reader will find openness at a linear or at a more general rate,
local covering, global covering, openness restricted to given sets,
point based openness, linear semiopenness, and so on.

In the present paper, a strong variant of the notion of openness
at a linear rate in a metric space setting is considered, which
applies only to multivalued mappings. Roughly speaking, such a
property postulates that the whole enlargement of images through
a given mapping are covered by the image of a single element near
the reference point, instead of by the image of an entire ball
around it. Such a requirement clearly imposes
severe restrictions on the covering behaviour of a
set-valued mapping, yet it happens to be fulfilled in various contexts,
which are relevant to variational analysis and optimization.
Furthermore, as shown in the present study, it exhibits nice
robustness features in the presence of various types of
perturbations.
Other motivations for the interest in this strong covering behaviour
come from set-inclusion problems, namely generalized
equations where the inclusion of a single element into images
of a given multifunction is replaced by the inclusion of an entire
set. With respect to such kind of problems, it seems that
solvability and solution stability can be hardly approached
as far as working with conventional covering notions.

The contents of the paper are organized as follows. In the rest
of the current section, basic notations, preliminary notions
and related facts, that will be employed throughout the paper,
are recalled. In Section \ref{Sect:2} the main covering property
under study is introduced. Several contexts in which it emerges
are discussed. In particular, the class of closed convex processes
having this behaviour is singled out. Then, conditions for
such a property to hold are established in a metric space setting.
Section \ref{Sect:3} is devoted to explore some applications to the
existence of set-inclusion points, with related error bound
estimates, and to the exact penalization of constrained optimization
problems.

Whenever $x$ is an element of a metric space $(X,d)$ and
$r$ is a positive real, $\ball{x}{r}=\{z\in X:\ d(z,x)\le r\}$ denotes
the closed ball with center $x$ and radius $r$. By $\dist{x}{S}=
\inf_{z\in S}d(z,x)$ the distance of $x$ from a subset $S\subseteq
X$ is denoted, with the convention that $\dist{x}{\varnothing}
=+\infty$. The $r$-enlargement of a set $S\subseteq X$ is indicated by
$\ball{S}{r}=\{x\in X:\ \dist{x}{S}\le r\}$. Given sets
$A,\, B\subseteq X$, the excess of $A$ over $B$ is indicated by
$\exc{A}{B}=\sup_{a\in A}\dist{a}{B}$, while the Hausdorff distance of
$A$ and $B$ by $\Haus{A}{B}=\max\{\exc{A}{B},\exc{B}{A}\}$.
Recall that a set-valued mapping $\Phi:X\rightrightarrows Y$
between metric spaces is said to be Lipschitz on $X$ with
constant $l\ge 0$ provided
$$
   \exc{\Phi(x_1)}{\Phi(x_2)}\le ld(x_1,x_2),\quad\forall
   x_1,\, x_2\in X.
$$
If the above inequality is satisfied in a neighbourhood of a
given point $\bar x\in X$, $\Phi$ is said to be locally Lipschitz
around $\bar x$.
A set-valued mapping $\Psi:X\rightrightarrows Y$
between metric spaces is called Hausdorff upper semicontinuous
(henceforth, u.s.c.) at $x_0\in X$ is for every $\epsilon>0$ there
exists $\delta>0$ such that $\Psi(x)\subseteq\ball{\Psi(x_0)}{\epsilon}$,
for every $x\in\ball{x_0}{\delta}$.
The domain and the graph of $\Psi:X\rightrightarrows Y$ are denoted
by $\dom\Psi$ and $\grph{\Psi}$, respectively. Throughout the paper,
any mapping $\Psi:X\rightrightarrows Y$ will be assumed to have
$\dom\Psi=X$ and to take closed values, unless otherwise stated.
In any vector space, the null element is marked by $\nullv$, and
the related notations $\B=\ball{\nullv}{1}$ and $\Sph=\fron\B=\B\backslash
\inte\B$ are adopted, where $\inte$ and $\fron$ indicate the topological
interior and boundary of a given set, respectively.

\begin{remark}    \label{rem:Lipexc}
In the sequel, the following consequence of the Lipschitz
property of a set-valued mapping on its excess function will
be used: if $\Phi$
is Lipschitz on $X$ with constant $l$, then for any nonempty
set $S\subseteq Y$, the function $x\mapsto\exc{\Phi(x)}{S}$
is Lipschitz on $X$ with the same constant $l$. Indeed, for every $x_1,
\,x_2\in X$ it is true that
\begin{eqnarray*}
  \exc{\Phi(x_2)}{S} &=& \sup_{y\in\Phi(x_2)}\dist{y}{S}\le
  \sup_{y\in\ball{\Phi(x_1)}{l d(x_1,x_2)}}\dist{y}{S}  \\
  &\le& \sup_{y\in\ball{\Phi(x_1)}{l d(x_1,x_2)}}\dist{y}{\Phi(x_1)}+
   \exc{\Phi(x_1)}{S}   \\
  &\le & l d(x_1,x_2)+ \exc{\Phi(x_1)}{S}.
\end{eqnarray*}
\end{remark}

Another property of the excess function associated with a pair
of set-valued mappings, that will be used in the sequel, is stated
next.

\begin{lemma}      \label{lem:lscexc}
Let $\Psi:X\rightrightarrows Y$ and $\Phi:X\rightrightarrows Y$
be given set-valued mappings between metric spaces. Suppose
that:
\begin{enumerate}

\item[(i)] $\Psi$ is Hausdorff u.s.c. at $x_0\in X$;

\item[(ii)] $\Phi$ is Lipschitz on $X$.

\end{enumerate}

\noindent Then, the function $\excf{\Phi}{\Psi}:X\longrightarrow
[0,+\infty)$ defined as
$$
    \excf{\Phi}{\Psi}(x)=\exc{\Phi(x)}{\Psi(x)},\quad x\in X,
$$
is lower semicontinuous (for short, l.s.c.) at $x_0\in X$.
\end{lemma}

\begin{proof}
Since $\excf{\Phi}{\Psi}$ acts on a metric space, it suffices to
show that for every sequence $(x_n)_{n\in\N}$ in $X$, with
$x_n\to x_0$ as $n\to\infty$, it results in 
\begin{eqnarray}   \label{in:lscexcf}
   \excf{\Phi}{\Psi}(x_0)\le\liminf_{n\to\infty}\excf{\Phi}{\Psi}(x_n).
\end{eqnarray}
Fix an arbitrary
$\epsilon>0$. By Hausdorff upper semicontinuity of $\Psi$ at $x_0$,
corresponding to $\epsilon$ there exists $\delta>0$ such that
$$
    \Psi(x)\subseteq\ball{\Psi(x_0)}{\epsilon},\quad\forall x\in
    \ball{x_0}{\delta}.
$$
As $x_n\to x_0$ as $n\to\infty$, there exists $n_\epsilon\in\N$
such that
$$
    \Psi(x_n)\subseteq\ball{\Psi(x_0)}{\epsilon},\quad\forall n\in\N,
    \hbox{ provided that } n\ge n_\epsilon.
$$
Therefore, one obtains that for any $y\in Y$ it holds
$$
   \dist{y}{\Psi(x_0)}\le\dist{y}{\Psi(x_n)}+\exc{\Psi(x_n)}{\Psi(x_0)}\le
   \dist{y}{\Psi(x_n)}+\epsilon
$$
for every $n\in\N$, with $n\ge n_\epsilon$. 
Now, by exploiting the Lipschitz continuity of the function $x
\mapsto \exc{\Phi(x)}{\Psi(x_0)}$ (recall Remark \ref{rem:Lipexc}),
for some $l\ge 0$ one obtains
\begin{eqnarray*}
  \excf{\Phi}{\Psi}(x_n) &=& \sup_{y\in\Phi(x_n)}\dist{y}{\Psi(x_n)}
  \ge\sup_{y\in\Phi(x_n)}\dist{y}{\Psi(x_0)}-\epsilon \\
   &\ge& \exc{\Phi(x_0)}{\Psi(x_0)}-ld(x_n,x_0)-\epsilon  \\
   &= &\excf{\Phi}{\Psi}(x_0)-ld(x_n,x_0)-\epsilon.
\end{eqnarray*}
It follows
$$
   \liminf_{n\to\infty}\excf{\Phi}{\Psi}(x_n)\ge\excf{\Phi}{\Psi}(x_0)
   -\epsilon.
$$
By arbitrariness of $\epsilon>0$, this shows the validity of
$(\ref{in:lscexcf})$,  thereby completing the proof.
\end{proof}


\section{Set-covering mappings}\label{Sect:2}

By a global covering behaviour of a multifunction $\Psi:X\rightrightarrows
Y$ acting between metric spaces the following property is
usually meant: there exists a constant $\alpha>0$ such that
\begin{equation}   \label{in:cov}
  \ball{\Psi(x)}{\alpha r}\subseteq \Psi(\ball{x}{r}),\quad
  \forall x\in X,\ \forall r>0
\end{equation}
(see, for instance,  \cite{Arut07,ArAvGeDmOb09,Dmit05}).

The main notion here under study comes up as a strong variant of
the above property, as stated below.

\begin{definition}     \label{def:setcov}
A set-valued mapping $\Psi:X\rightrightarrows Y$ between metric spaces
is said to be {\it set-covering} on $X$ with constant $\alpha$ if there
exists a positive real $\alpha$ such that
\begin{equation}   \label{in:defsetcov}
    \forall x\in X,\ \forall r>0\ 
   \exists u\in\ball{x}{r}\hbox{ such that }
   \ball{\Psi(x)}{\alpha r}\subseteq \Psi(u).
\end{equation}
\end{definition}

From Definition \ref{def:setcov} it is clear that, if a set-valued
mapping is set-covering with constant $\alpha$, then it is
covering in the sense of $(\ref{in:cov})$, with the same constant.
The converse is not true, as readily illustrated in the counterexamples
below. Other immediate consequences of Definition \ref{def:setcov}
are the facts that $\Psi$ is onto and that densely on $X$ it
takes values with nonempty interior.

\begin{example}      \label{ex:covnotsetcov}
Let $X=\R$ and $Y=\R^2$ be endowed with their usual (Euclidean)
metric structure. Consider the set-valued mapping $\Psi:\R\rightrightarrows
\R^2$ given by $\Psi(x)=|x|\Sph$. It is not difficult to see that $\Psi$
is covering on $\R$ with constant $\alpha=1$, whereas it fails to fulfil
Definition \ref{def:setcov} for any $\alpha>0$, as it is $\inte\Psi(x)
=\varnothing$ for every $x\in\R$.
Again, the mapping $\Psi:\R\rightrightarrows\R^2$ given by
$\Psi(x)=\ball{x}{1}$ is covering with constant $1$, but is not
set-covering, even if it takes images with nonempty interior.
This second mapping shows that, while Definition \ref{def:setcov}
forces a mapping to have images with nonempty interior in a
dense subset of $X$, this topological requirement is only
necessary.
\end{example}

\begin{example}     \label{ex:metdilat}
Let $\delta:X\longrightarrow [0,+\infty)$ be a function defined
on a metric space $(X,d)$ and satisfying the condition
\begin{eqnarray}      \label{in:deltacond}
     \inf_{x\in X} \inf_{r>0}\sup_{u\in\fron\ball{x}{r}}
   {\delta(u)-\delta(x)\over d(u,x)}=\alpha_\delta>0,
\end{eqnarray}
and let $(Y,d)$ be a metric space. For any fixed $y_0\in Y$,
the set-valued mapping $\Psi:X\rightrightarrows Y$, defined by
$$
     \Psi(x)=\ball{y_0}{\delta(x)},
$$
is set-covering on $X$ with any constant $\alpha\in
(0,\alpha_\delta)$. Indeed, fixed such an $\alpha$, let $x\in X$ and
$r>0$. Take $\epsilon>0$ in such a way that $\alpha+2\epsilon<
\alpha_\delta$. By condition $(\ref{in:deltacond})$, corresponding
to $\epsilon$, there exists $u\in\fron\ball{x}{r}$ such that
$\delta(u)\ge\delta(x)+(\alpha_\delta-\epsilon)r$. Thus, if $y\in
\ball{\Psi(x)}{\alpha r}$, that is $\dist{y}{\ball{y_0}{\delta(x)}}
\le\alpha r$, it results in
$$
    d(y,y_0)<\delta(x)+(\alpha+\epsilon)r<\delta(x)+
     (\alpha_\delta-\epsilon)r\le\delta(u).
$$
This means that $y\in\ball{y_0}{\delta(u)}=\Psi(u)$. Since
$u\in\ball{x}{r}$, the requirement $(\ref{in:defsetcov})$ is fulfilled.
Notice that, whenever $X$ is in particular a normed space, taking
$\delta(\cdot)=\|\cdot\|$, one finds $\alpha_{\|\cdot\|}=1$,
so condition $(\ref{in:deltacond})$ is valid.
\end{example}

Below some natural circumstances, in which the covering behaviour
formalized in Definition \ref{def:setcov} appear, are presented.

\begin{example}
(Solution mappings to systems of sublinear inequalities)
Let $X=\R^n$ be metrized with the norm $\|\cdot\|_\infty$ and let
$Y=\R^m$ be endowed with its usual (Euclidean) metric structure.
Suppose that $n$ functions $p_i:\R^m\longrightarrow\R$,
with $i=1,\dots,n$, are given, which are sublinear on $\R^m$,
i.e. such that
$$
    p_i(\nullv)=0,\quad p_i(t y)=tp_i(y),\quad\forall t>0,\, \forall y
   \in\R^m, \quad \hbox{ and $p_i$ convex on } \R^m.  
$$
Set
$$
    \|p_i\|_*=\max\{\|y^*\|:\ y^*\in\partial p_i(\nullv)\},\quad
    i=1,\dots,n,
$$
where $\partial p_i(\nullv)$ denotes the subdifferential of $p_i$ at
$\nullv$ in the sense of convex analysis, and
$$
    \|p\|_*=\max_{i=1,\dots,n}\|p_i\|_*.
$$
Notice that $\|p_i\|_*$ and $\|p\|_*$ are well defined and finite, as each $\partial
p_i(\nullv)$ is a nonempty compact subset of $\R^m$. In what follows, it is
assumed that $\|p\|_*>0$, as the case $\|p\|_*=0$ leads to $p_i
\equiv 0$ for every $i=1,\dots,n$, which is of minor interest here.
Consider the solution mapping $\Psi:\R^n\rightrightarrows\R^m$
associated with a parameterized inequality system involving
functions $p_i$ as follows
$$
    \Psi(x)=\{y\in\R^m:\ p_i(y)\le|x_i|,\quad\forall i=1,\dots,n\}.
$$
$\Psi$ clearly takes nonempty closed and convex values. Let us show that
$\Psi$ is set-covering, with constant $\alpha=1/\|p\|_*$.
To do so, fixed $x\in\R^n$ and $r>0$, take $y\in\ball{\Psi(x)}
{r/\|p\|_*}$. This implies the existence of $v\in\Psi(x)$ such that
$\|y-v\|\le r/\|p\|_*$. Thus, it must be
$$
    p_i(v)\le |x_i|,\quad\forall i=1,\dots,n.
$$
By the sublinearity of each $p_i$, one has
\begin{eqnarray*}
   p_i(y)\le p_i(y-v)+p_i(v)\le\|p\|_*\|y-v\|+|x_i|\le r+|x_i|,
   \ \forall i=1,\dots,n.
\end{eqnarray*}
Consequently, by defining $u\in\R^n$ as
\begin{eqnarray*}
   u_i=\left\{ \begin{array}{ll}
             r+x_i, &\quad\hbox{ if } x_i\ge 0, \\
           -r+x_i,   &\quad\hbox{ if } x_i<0,
   \end{array} \right.
\end{eqnarray*}
it results in
$$
    p_i(y)\le |u_i|,\quad\forall i=1,\dots,n,
$$
and hence
$$
    \balll{\Psi(x)}{{r\over\|p\|_*}}\subseteq\Psi(u),
$$
with $\|u-x\|_\infty\le r$. If $\R^n$ is remetrized through
another (equivalent) norm, $\Psi$ still remains set-covering, but with
a different constant.
\end{example}

\begin{example}(Epigraphical mappings in partially ordered
normed spaces) Let $(\X,\|\cdot\|)$ and $(\Y,\|\cdot\|)$ be real
normed spaces. Suppose that on $\Y$ a partial order relation
$\por{\Y}$ is induced by a closed, convex, pointed cone $\Y_+
\subseteq\Y$, in the sense that $y_1\por{\Y} y_2$ iff $y_2-y_1
\in\Y_+$.
Let us introduce the following assumption on the interplay
bewteen the partial order and the metric structure on $\Y$:
\begin{eqnarray}    \label{ass:ordermetric}
   \qquad\exists\gamma\in [1,+\infty):\ 
    \forall r>0\ \exists\check{y}\in\Y:\ \check{y}\por{\Y}y,\ \forall
    y\in r\B \ \hbox{ and }\ \|\check{y}\|=\gamma r.
\end{eqnarray}
Of course, because of the linearity of the partial order,
assumption $(\ref{ass:ordermetric})$ applies also to balls
with center at each point of $\Y$.
Such an assumption is verified, for instance, if $\Y=\R^m$, $\Y_+=
\R^m_+$, and $\|\cdot\|_p$ is a $p$-norm, with $p\in [1,+\infty)$,
by the constant $\gamma=m^{1/p}$. Otherwise, if $\Y=\R^m$ is normed by
$\|\cdot\|_\infty$, one has $\gamma=1$. If $\Y=C([0,1])$, $\Y_+=
\{x\in C([0,1]):\ x(t)\ge 0,\quad\forall t\in [0,1]\}$ and the metric
structure on $\Y$ is induced by the norm $\|\cdot\|_\infty$, then again
assumption $(\ref{ass:ordermetric})$ is true with $\gamma=1$.
Instead, if the space $(\ell_p(\N),\|\cdot\|_p)$, with $p\in [1,+\infty)$ is partially
ordered by the componentwise order relation, assumption
$(\ref{ass:ordermetric})$ fails to be verified. If $(\ell_\infty(\N),
\|\cdot\|_\infty)$ is partially ordered by the same order relation,
the above assumption is verified with $\gamma=1$.

Now, let $f:\X\longrightarrow\Y$ be a mapping covering on $\X$
with constant $\alpha>0$, i.e. such that
\begin{eqnarray}    \label{in:covsinglemap}
    f(\ball{x}{r})\supseteq\ball{f(x)}{\alpha r},\quad\forall
    x\in\X,\ \forall r>0.
\end{eqnarray}
It is possible to show that its epigraphical set-valued mapping
${\rm epi}_f:\X\rightrightarrows\Y$, which is defined as
$$
    {\rm epi}_f(x)=f(x)+\Y_+
$$
or, equivalently, by the condition
$$
    \grph{{\rm epi}_f}=\epi{f},
$$
where $\epi{f}=\{(x,y)\in\X\times\Y:\ f(x)\por{\Y}y\}$,
is set-covering in the sense of Definition \ref{def:setcov}, with
constant $\alpha/\gamma$. To see this, fix $x\in\X$ and $r>0$,
and consider an element $\check{y}\in\Y$ as in $(\ref{ass:ordermetric})$,
namely such that
$$
    \check{y}\por{\Y}y,\quad\forall y\in\ball{f(x)}{\alpha r}
    \quad\hbox{ and }\quad \|\check{y}-f(x)\|=\gamma\alpha r.
$$
This implies that
$$
    \ball{f(x)}{\alpha r}\subseteq\check{y}+\Y_+.
$$
Since it is $\check{y}\in\ball{f(x)}{\alpha\gamma r}$, then by
virtue of $(\ref{in:covsinglemap})$ there must exist $u\in
\ball{x}{\gamma r}$ such that $\check{y}=f(u)$. Thus, letting
$\rho=\gamma r$, one obtains
\begin{eqnarray*}
   \balll{{\rm epi}_f(x)}{{\alpha\over\gamma}\rho} &=&
   \ball{f(x)+\Y_+}{\alpha r}\subseteq\ball{f(x)}{\alpha r}+\Y_+ \\
   &\subseteq &\check{y}+\Y_+=f(u)+\Y_+={\rm epi}_f(u),
\end{eqnarray*}
with $u\in\ball{x}{\rho}$. Whenever $(\X,\|\cdot\|)$ and $(\Y,
\|\cdot\|)$ are, in particular, Banach spaces and $f:\X\longrightarrow
\Y$ is a bounded linear operator, then according to the Banach-Schauder
open mapping theorem $f$ is known to be
covering on $\X$ iff it is onto. In such an event, the quantity
$$
    \|f^{-1}\|_-=\sup_{y\in\B}\inf\{\|x\|:\ f(x)=y\}=\sup_{y\in\B}
   \dist{\nullv}{f^{-1}(y)},
$$
where $f^{-1}:\Y\rightrightarrows\X$ is the (generally) set-valued
inverse mapping of $f$,
is a positive element of $\R$ and, as a covering constant, one
can take $\alpha=1/\|f^{-1}\|_-$. Thus the epigraphical mapping
of a bounded linear operator, which is onto, is set-covering with
constant $1/\gamma\|f^{-1}\|_-$.
\end{example}

To the aim of providing further examples of classes
of set-covering mappings,  let us focus now on convex
processes. The idea of a convex process is due to R.T. Rockafellar
(see \cite{Rock70}) and emerges when dealing with derivatives
of set-valued mappings or with certain constraint systems
arising in optimization problems.
After him, a set-valued mapping $\Theta:\X\rightrightarrows
\Y$ between normed spaces is said to be a convex process
if $\grph{\Theta}$ is a convex cone of $\X\times\Y$ with apex
at the null vector, or, equivalently, iff $\Theta$ satisfies all
the following three requirements:
\begin{enumerate}

\item[(i)] $\nullv\in\Theta(\nullv)$;

\item[(ii)] $\Theta(\lambda x)=\lambda\Theta(x),\quad\forall\lambda>0,
\ \forall x\in\X$;

\item[(iii)] $\Theta(x_1)+\Theta(x_2)\subseteq\Theta(x_1+x_2),\quad
\forall x_1,\, x_2\in\X$.

\end{enumerate}
Clearly $\Theta$ is a convex process iff $\Theta^{-1}$ is so.
Further, a convex process is said to be closed provided that so
is its graph. A way to approach the study of the covering behaviour
of convex processes is through the notion of openness at
$\nullv$. According to \cite{Robi72}, a convex process $\Theta$ is
said to be {\it open} at $\nullv$ if there exists $\alpha>0$
such that
\begin{eqnarray}      \label{in:defopen}
    \Theta(\inte\B)\supseteq\inte\alpha\B.
\end{eqnarray}
Such a condition has been characterized in terms of finiteness of
the inner norm of the inverse mapping. More precisely, defined
the inner norm of a mapping $\Theta:\X\rightrightarrows\Y$ as
$$
    \|\Theta\|_-=\sup_{x\in\dom\Theta\cap\B}\inf
    \{\|y\|:\ y\in\Theta(x)\},
$$
it has been established that $\Theta$ is open at $\nullv$
iff $\|\Theta^{-1}\|_-<+\infty$ and, as a constant appearing
in $(\ref{in:defopen})$, it is possible to take any value
$\alpha\in (0,1/\|\Theta^{-1}\|_-^{-1})$ (see \cite{Robi72}).
Whenever $(\X,\|\cdot\|)$ and $(\Y,\|\cdot\|)$ are, in
particular, Banach spaces, a sufficient condition for a closed,
convex process $\Theta:\X\rightrightarrows\Y$ to be open
at $\nullv$ is that $\Theta$ is onto (open mapping theorem
for convex processes, see \cite{AubFra90,DonRoc14,Robi72}).

The next result and the subsequent related remark show that
a proper subclass of onto and closed convex processes can
be found, whose elements are set-covering mappings.

\begin{proposition}      \label{pro:condsetcovcp}
Let $\Theta:\X\rightrightarrows\Y$ be a closed, convex process
between Banach spaces. If the following condition holds
\begin{eqnarray}     \label{in:condsetcovcp}
    \exists\alpha>0,\quad \exists u\in\B\quad\hbox{such that}\quad
    \Theta(u)\supseteq\inte\alpha\B,
\end{eqnarray}
then $\Theta$ is set-covering with any constant $\tilde\alpha\in
(0,\alpha)$. Vice versa, if $\Theta$ is set-covering with a constant
$\alpha>0$, then condition $(\ref{in:condsetcovcp})$ holds.
\end{proposition}

\begin{proof}
Assume first that condition $(\ref{in:condsetcovcp})$ holds
true. Fix arbitrary $x\in\X$, $r>0$ and $\tilde\alpha\in (0,\alpha)$,
take $y\in\ball{\Theta(x)}{\tilde\alpha r}$ and pick $\epsilon>0$
in such a way that $\tilde\alpha(1+\epsilon)<\alpha$. This implies
that there exists of $v\in\Theta(x)$ such that
$$
    \|y-v\|<\tilde\alpha(1+\epsilon)r,
$$
that is $y-v\in\inte\tilde\alpha(1+\epsilon)r\B$. Since, as a convex
process, $\Theta$ is positively homogeneous,
condition $(\ref{in:condsetcovcp})$ entails the existence of
$u\in r\B$ such that
$$
   \Theta(u)\supseteq\inte\alpha r\B\supseteq\inte
  \tilde\alpha(1+\epsilon)r\B.
$$
Consequently, one has
$$
   y=(y-v)+v\in\inte\alpha r\B+\Theta(x)\subseteq\Theta(u)+\Theta(x)
   \subseteq\Theta(u+x),
$$
where $u+x\in\ball{x}{r}$. It should be noticed that, according
to the inclusion $(\ref{in:condsetcovcp})$, the element $u$ in
$r\B$ does not depend neither on $y$ nor on $v$.

Vice versa, if choosing $x=\nullv$ and $r=1$ in Definition
\ref{def:setcov}, since it is $\nullv\in\Theta(\nullv)$, one finds
$$
   \inte\alpha\B\subseteq\ball{\Theta(\nullv)}{\alpha}\subseteq
   \Theta(u)
$$
for some $u\in\B$ and hence condition $(\ref{in:condsetcovcp})$
is verified at once.
\end{proof}

\begin{remark}
Condition $(\ref{in:condsetcovcp})$ is a quantitative requirement
about the surjective behaviour of $\Theta$ that can be expressed
in merely topological terms as
\begin{eqnarray}   \label{ne:intnonempty}
    \inte\Theta(\nullv)\ne\varnothing.
\end{eqnarray}
Indeed, such a condition obviously implies $(\ref{in:condsetcovcp})$.
Vice versa, if $u$ and $\alpha$ are as in $(\ref{in:condsetcovcp})$,
one has
$$
   \Theta(\nullv)=\Theta(u)+\Theta(-u)\supseteq\inte\alpha\B+
   \Theta(-u),
$$
wherefrom the interior noneptiness condition in
$(\ref{ne:intnonempty})$ follows. It is worth noting that condition
$(\ref{in:condsetcovcp})$ is essentially stronger than openness at
$\nullv$. In other words, the latter is
sufficient for a closed, convex process to be covering, whereas
it does not in the case of the set-covering property, even in the
case of convex processes.
This occurence is illustrated in the next example.
\end{remark}

\begin{example}
Consider the Banach space $\X=\Y=(\ell_p(\N),\|\cdot\|_p)$,
with $p\in[1,+\infty)$ and define $\Theta:\ell_p(\N)\rightrightarrows
\ell_p(\N)$ as
$$
    \Theta(x)=x+\X_+,
$$
where $\X_+=\{x=(\xi_n)_{n\in\N}:\ \xi_n\ge 0,\quad\forall n\in\N\}$.
It is not difficult to check that $\Theta$ is a closed, convex
process and it is clear that $\Theta$ is also onto. Therefore
$\Theta$ is open at $\nullv$ and hence covering.
Nevertheless, in the light of Proposition \ref{pro:condsetcovcp}
$\Theta$ fails to be set-covering. Indeed, it is $\Theta(\nullv)=\X_+$
and such a cone is well known to have empty topological
interior (see, for instance, \cite{Jahn86}).
\end{example}

In the rest of this section, the context is again that of
set-valued mappings between metric spaces.
Given a set-valued mapping $\Psi:X\rightrightarrows Y$
and a nonempty set $S\subseteq Y$, there are two natural
notions of inverse image of $S$ through $\Psi$, which are
known in set-valued analysis as upper (or strong) 
and lower (or weak) inverse, respectively. Here yet another
notion is considered, which is defined as follows
$$
     \Inv{\Psi}(S)=\{x\in X:\ S\subseteq\Psi(x)\}.
$$
Letting $S$ to vary in $2^Y$, one obtains from $\Psi$
a set-valued mapping $\Inv{\Psi}:2^Y\rightrightarrows X$.
Such a mapping can be viewed as solution mapping of a set-inclusion,
where the set $S$ plays the role of a parameter. As one expects,
a reformulation of the set-covering property can be
expressed in terms of error bound for the solution mapping
associated with such a parameterized set-inclusion problem.

\begin{proposition}     \label{pro:setmrchar}
Let $\Psi:X\rightrightarrows Y$ be a  set-valued mapping between
metric spaces.
\begin{enumerate}

\item[(i)] If $\Psi$ is set-covering with constant $\alpha>0$, then
\begin{eqnarray}    \label{in:setmrInv}
    \dist{x}{\Inv{\Psi}(S)}\le{1\over\alpha}\exc{S}{\Psi(x)},\quad
    \forall x\in X,\ \forall S\in 2^Y.
\end{eqnarray}

\item[(ii)] If inequality $(\ref{in:setmrInv})$ holds, then $\Psi$
is set-covering with any constant $\tilde\alpha\in (0,\alpha)$. 
\end{enumerate}
\end{proposition}

\begin{proof}
(i) Assume that $x\in X$ and $S\subseteq Y$ are arbitrary.
If  $\exc{S}{\Psi(x)}=+\infty$, then $(\ref{in:setmrInv})$ trivially
holds. In this case, it may happen that $\Inv{\Psi}(S)=\varnothing$. 
 Thus, let us pass to consider the case $r=\exc{S}{\Psi(x)}
<+\infty$. Since $\Psi$ takes closed values, if $r=0$
one has $S\subseteq\Psi(x)$ and hence $x\in\Inv{\Psi}(S)$. If $r>0$,
since $S\subseteq\ball{\Psi(x)}{r}$, according to Definition \ref{def:setcov},
there exists $u\in\ball{x}{r/\alpha}$ such that $S\subseteq\Psi(u)$.
It follows
$$
   \dist{x}{\Inv{\Psi}(S)}\le d(x,u)\le{1\over\alpha}\exc{S}{\Psi(x)}.
$$
(ii) Fix any $\tilde\alpha\in (0,\alpha)$. To see that in this case
Definition \ref{def:setcov} is satisfied, it suffices to take
$S=\ball{\Psi(x)}{\tilde\alpha r}$ and to observe that, with this
choice, it is $\exc{S}{\Psi(x)}\le\tilde\alpha r<+\infty$.
Therefore, from inequality $(\ref{in:setmrInv})$ one gets
$$
    \dist{x}{\Inv{\Psi}(S)}\le {1\over\alpha}\tilde\alpha r<r,
$$
which means that there exists $u\in\ball{x}{r}$ such that
$\Psi(u)\supseteq \ball{\Psi(x)}{\tilde\alpha r}$.
\end{proof}

\begin{remark}
It is proper to warn the reader that, in general, it may happen thatù
$\dom\Inv{\Psi}\ne 2^Y$.
From inequality $(\ref{in:setmrInv})$ it follows that, whenever
$\Psi$ is set-covering and $S\subseteq Y$ is such that $\exc{S}{\Psi(x)}
<+\infty$, then it must be ${\Inv{\Psi}(S)}\ne\varnothing$. In particular,
$\dom\Inv{\Psi}$ includes all bounded subsetes of $Y$.
Let us denote by ${\mathcal B}(Y)$ the collection of all such
subsets of $Y$.
\end{remark}

The next step consists in linking the set-covering property
with the Lipschitzian behaviour
of the mapping $\Inv{\Psi}$. To this aim, the set ${\mathcal B}(Y)$
is equipped with the Hausdorff distance.

\begin{proposition}
Let $\Psi:X\rightrightarrows Y$ be a  set-valued mapping between
metric spaces. If $\Psi$ is set-covering with constant $\alpha>0$, then
$\Inv{\Psi}:{\mathcal B}(Y)\rightrightarrows X$ is Lipschitz with
constant $1/\alpha$. Vice versa, given $\Psi:X\longrightarrow
{\mathcal B}(Y)$, if $\Inv{\Psi}:{\mathcal B}(Y)\rightrightarrows X$
is Lipschitz with constant $0<l<+\infty$, then $\Psi$ is set-covering
on $X$ with any constant $\alpha\in (0,1/l)$.
\end{proposition}

\begin{proof}
Consider a pair of elements $A,\, B\in{\mathcal B}(Y)$. From
assertion (i) in Proposition \ref{pro:setmrchar}, one has
$$
    \dist{x}{\Inv{\Psi}(B)}\le \alpha^{-1}\exc{B}{\Psi(x)},
    \quad\forall x\in X.
$$
Thus, for all those $x\in\Inv{\Psi}(A)$, i.e. such that $A\subseteq
\Psi(x)$, one finds
$$
    \dist{x}{\Inv{\Psi}(B)}\le \alpha^{-1}\exc{B}{A},
$$
whence
$$
   \exc{\Inv{\Psi}(A)}{\Inv{\Psi}(B)}=\sup_{x\in\Inv{\Psi}(A)}
   \dist{x}{\Inv{\Psi}(B)}\le\alpha^{-1}\exc{B}{A}.
$$
To achieve the inequality
$$
   \Haus{\Inv{\Psi}(A)}{\Inv{\Psi}(B)}\le\alpha^{-1}\Haus{A}{B},
$$
it suffices to interchange the role of $A$ and $B$.

To prove the second assertion in the thesis, let $x\in X$ and $r>0$
be arbitrary. Since now $\Psi(x)\in\mathcal{B}(Y)$, the same is true
for $\ball{\Psi(x)}{\alpha r}$. By the Lipschitz continuity of $\Inv{\Psi}$
with constant $l$, if taking any $\alpha\in (0,1/l)$ one finds

\begin{eqnarray*}
  \Haus{\Inv{\Psi}(\Psi(x))}{\Inv{\Psi}(\ball{\Psi(x)}{\alpha r})}&\le& l
  \Haus{\Psi(x)}{\ball{\Psi(x)}{\alpha r}}  \\
    &= & l\exc{\ball{\Psi(x)}{\alpha r}}{\Psi(x)}<r.
\end{eqnarray*}
On the other hand, observe that it is
\begin{eqnarray*}
   \Inv{\Psi}(\Psi(x)) &=& \{u\in X:\ \Psi(x)\subseteq\Psi(u)\}\supseteq
   \Inv{\Psi}(\ball{\Psi(x)}{\alpha r}) \\
   &=& \{u\in X:\ \ball{\Psi(x)}{\alpha r}\subseteq\Phi(u)\}.
\end{eqnarray*}
Thus, the last inequality amounts to say that
$$
    \exc{\Inv{\Psi}(\Psi(x))}{\Inv{\Psi}(\ball{\Psi(x)}{\alpha r})}<r,
$$
wherefrom, as in particular it is $x\in\Inv{\Psi}(\Psi(x))$, one
obtains
$$
   \dist{x}{\{u\in X:\ \ball{\Psi(x)}{\alpha r}\subseteq\Phi(u)\}}
   <r.
$$
The last inequality implies the existence of $u\in\ball{x}{r}$
with the property that $\ball{\Psi(x)}{\alpha r}\subseteq\Psi(u)$, thereby
showing that $\Psi$ is set-covering with constant $\alpha$.
\end{proof}

The rest of the present section is devoted to illustrate some
robustness features of the set-covering property in the
presence of various perturbations.

\begin{proposition}
Let $\Psi:X\rightrightarrows Y$ be set-covering on $X$ with constant
$\alpha$ and let $g:Y\longrightarrow Z$ be covering on $Y$ with
constant $\beta$. Then, their composition $g\circ\Psi:X\rightrightarrows
Z$ is set-covering with any constant $\gamma\in (0,\alpha\beta)$.
\end{proposition}

\begin{proof}
By the set-covering property of $\Psi$, corresponding to $x\in X$
and $r>0$, there exists $u\in\ball{x}{r}$ such that $\ball{\Psi(x)}{\alpha
r}\subseteq\Psi(u)$, whence it follows
\begin{eqnarray}    \label{in:gsetcovPsi}
   g(\ball{\Psi(x)}{\alpha r})\subseteq (g\circ\Psi)(u).
\end{eqnarray}
By the covering property on $Y$ of $g$, one has
\begin{eqnarray}    \label{in:BggB}
    \ball{g(y)}{\alpha\beta r}\subseteq g(\ball{y}{\alpha r}),
    \quad\forall y\in Y,\ \forall r>0.
\end{eqnarray}
Notice that, since for any $z\in Z$ it is
$$
    \dist{z}{(g\circ\Psi)(x)}=\inf_{y\in\Psi(x)}d(z,g(y)),
$$
it holds

$$
  \ball{(g\circ\Psi)(x)}{\gamma r}\subseteq\bigcup_{y\in\Psi(x)}
  \ball{g(y)}{\alpha\beta r}.
$$
Thus, in the light of inclusion $(\ref{in:BggB})$, one obtains
\begin{eqnarray*}
    \ball{(g\circ\Psi)(x)}{\gamma r}&\subseteq &
   \bigcup_{y\in\Psi(x)}g(\ball{y}{\alpha r})\subseteq
   g\left(\bigcup_{y\in\Psi(x)}\ball{y}{\alpha r}\right)    \\
   &\subseteq &  g(\ball{\Psi(x)}{\alpha r}).
\end{eqnarray*}
By recalling inclusion $(\ref{in:gsetcovPsi})$, one deduces that
$$
   \ball{(g\circ\Psi)(x)}{\gamma r}\subseteq(g\circ\Psi)(u).
$$
This completes the proof.
\end{proof}

The next proposition relates to a stability phenomenon regarding
set-covering, which can be observed to take place in the presence
of additive perturbations by single-valued Lipschitz mappings.
In doing so, along with the previous one, it provides as well
a tool for building further examples of classes of set-covering
mappings.

\begin{proposition}     \label{pro:staLippert}
Let $X$ be a metric space and let $Y$ be a vector space,
equipped with a shift invariant metric.
Let $\Psi:X\rightrightarrows Y$ and $g:X\longrightarrow Y$
be a set-valued and a single-valued mapping, respectively.
Suppose that $\Psi$ is set-covering on $X$ with constant
$\alpha>0$, whereas $g$ is Lipschitz on $X$ with constant
$\beta\in [0,\alpha)$. Then, the mapping $\Psi+g$ is set-covering
on $X$ with constant $\alpha-\beta$.
\end{proposition}

\begin{proof}
Fixed $x\in X$ and $r>0$, according to Definition \ref{def:setcov}
one has to show that there exists $u\in\ball{x}{r}$ such that
\begin{eqnarray}      \label{in:Psiplusgthesis}
    \ball{\Psi(x)+g(x)}{(\alpha-\beta)r}\subseteq\Psi(u)+g(u).
\end{eqnarray}
Take an arbitrary $y\in\ball{\Psi(x)+g(x)}{(\alpha-\beta)r}$. This
means that it is $\dist{y}{\Psi(x)+g(x)}\le(\alpha-\beta)r$ or, by
virtue of the shift invariance property of the metric on $Y$,
\begin{eqnarray}      \label{in:shiftgPsidist}
    \dist{y-g(x)}{\Psi(x)}\le(\alpha-\beta)r.
\end{eqnarray}
Since $g$ is Lipschitz on $X$ with constant $\beta$, one
has
$$
    d(g(z),g(x))\le\beta r,\quad\forall z\in\ball{x}{r}.
$$
Consequently, from inequality $(\ref{in:shiftgPsidist})$ one
obtains again by shift invariance
\begin{eqnarray*}
  \dist{y-g(z)}{\Psi(x)}&\le& d(y-g(z),y-g(x))+\dist{y-g(x)}{\Psi(x)} \\
  &\le &\beta r+(\alpha-\beta)r=\alpha r, \quad\forall z\in\ball{x}{r}.
\end{eqnarray*}
In other terms, it holds
$$
    y-g(z)\in\ball{\Psi(x)}{\alpha r},\quad\forall z\in\ball{x}{r}.
$$
By using the fact that $\Psi$ is set-covering on $X$ with constant
$\alpha>0$, one can state that there exists $u\in\ball{x}{r}$, such
that
$$
    y-g(z)\in\Psi(u),\quad\forall z\in\ball{x}{r}.
$$
The reader should notice that such an element as $u$ does not
depend neither on $z$ nor on $y$, because $\Psi(u)$ covers the
whole set $\ball{\Psi(x)}{\alpha r}$.
In particular, one has
$$
    y-g(u)\in\Psi(u),
$$
which implies that $y\in\Psi(u)+g(u)$. By arbitrariness of $y$
in $\ball{\Psi(x)+g(x)}{(\alpha-\beta)r}$, this proves the
inclusion in $(\ref{in:Psiplusgthesis})$, thereby completing
the proof.
\end{proof}

\begin{remark}
(i) The phenomenon described by Proposition \ref{pro:staLippert}
can be inserted in the framework of the stability analysis for
covering behaviours started with the well-known Milyutin theorem
(see, for instance, \cite{Arut07,ArAvGeDmOb09,Dmit05,DmMiOs80}).
Note that, in contrast to the latter, which refers to a traditional
covering behaviour, no completeness assumption is needed in the
case of set-covering.

(ii) As a comment to the assumptions of Proposition \ref{pro:staLippert},
it is to be pointed out that the shift invariance requirement on the metric
of $Y$ is not really restrictive. Indeed, according to a result due to S.
Kakutani, any linear metric space can be equivalently remetrized by
a shift invariant distance (see \cite{Role87}, Theorem 2.2.11).
\end{remark}

\begin{example}
Let $(X,d)$ be a metric space and let $(Y,d)$ be a vector space
endowed with a shift-invariant metric. Then, as a consequence of
Proposition \ref{pro:staLippert} and Example \ref{ex:metdilat},
if $\delta:X\longrightarrow
[0,+\infty)$ fulfils condition $(\ref{in:deltacond})$ and $g:X
\longrightarrow Y$ is a Lipschitz mapping with constant $\beta
<\alpha_\delta$, then the set-valued mapping $\Psi:X\rightrightarrows
Y$ given by
$$
    \Psi(x)=\ball{g(x)}{\delta(x)}=\ball{\nullv}{\delta(x)}+g(x)
$$
turns out to be set-covering with any constant $\alpha\in (0,
\alpha_\delta-\beta)$.
\end{example}


\section{Set-inclusion points of pairs of mappings and applications}     \label{Sect:3}

\subsection{Set-inclusion points}

The next definition introduces a very general problem, which can be
posed whenever any pair of multivalued mappings is given.

\begin{definition}
Given two set-valued mappings $\Psi:X\rightrightarrows Y$
and $\Phi:X\rightrightarrows Y$, an element $x\in X$ is called a
{\it set-inclusion point} of the (ordered) pair $(\Phi,\Psi)$ if
$$
    \Phi(x)\subseteq\Psi(x).
$$
Denote
$$
    \Inc{\Phi}{\Psi}=\{x\in X:\ \Phi(x)\subseteq\Psi(x)\}.
$$
\end{definition}

A set-inclusion point problem is an abstract formalism able to subsume
in its extreme generality
several specific problems, having or not a variational nature.
For instance, it enables one to embed equilibrium conditions, fixed
or coincidence point problems, generalized equations, by proper
choices of $\Phi$ and $\Psi$. Nonetheless, it is when $\Phi$ (and hence
$\Psi$) is actually a multivalued mapping that its peculiarity appears.
Besides, it is interesting to note that $\Inc{\Phi}{\Psi}$ can be
regarded as the set of all fixed points of the mapping
$\Inv{\Psi}\circ\Phi:X\rightrightarrows X$.

\begin{remark}    \label{rem:Incclosed}
It is worth noting that, as a straightforward consequence of Lemma
\ref{lem:lscexc}, the set $\Inc{\Phi}{\Psi}$ is closed whenever $\Psi$ is
Hausdorff u.s.c. on $X$ and $\Phi$ is Lipschitz on $X$. Indeed, if $\Inc{\Phi}{\Psi}
\ne\varnothing$ and $(x_n)_{n\in\N}$ is a sequence in $\Inc{\Phi}{\Psi}$
converging to $x_0$ as $n\to\infty$, one finds
$$
   0=\liminf_{n\to\infty}\excf{\Phi}{\Psi}(x_n)\ge\excf{\Phi}{\Psi}(x_0)
   \ge 0, 
$$
wherefrom, by closedness of $\Psi(x_0)$, it follows that
$x_0\in\Inc{\Phi}{\Psi}$.
\end{remark}

In what follows, pursuing a similar line of reasearch as in
\cite{Arut07,ArAvGeDmOb09},
the question of the solution existence of set-inclusion
points is analyzed in the general setting of multifunctions
between metric spaces. In particular,
the next result provides a sufficient condition for a set-inclusion
problem, involving a set-covering and a Lipschitz mappings,
to admit a solution, as well as an error bound for its
solution set.

\begin{theorem}      \label{thm:incpoint}
Let $\Psi:X\rightrightarrows Y$ and $\Phi:X\longrightarrow {\mathcal B}
(Y)$ be given set-valued mappings between metric spaces. Suppose
that:

\begin{enumerate}

\item[(i)] $(X,d)$ is metrically complete;

\item[(ii)] $\Psi$ is Hausdorff u.s.c. and set-covering on $X$,
with constant $\alpha>0$;

\item[(iii)] $\Phi$ is Lipschitz on $X$ with constant $\beta
\in [0,\alpha)$.

\end{enumerate}

\noindent Then, $\Inc{\Phi}{\Psi}\ne\varnothing$ and the following
estimate holds
\begin{eqnarray}      \label{in:Incdist}
   \dist{x}{\Inc{\Phi}{\Psi}}\le {\exc{\Phi(x)}{\Psi(x)}\over
   \alpha-\beta},\quad\forall x\in\ X.
\end{eqnarray}
\end{theorem}

\begin{proof}
The proof is based on a variational technique. Notice indeed
that, in order to prove the existence of
a set-inclusion point $\bar x\in X$ for the pair $\Psi$
and $\Phi$, it suffices to show
that the function $\excf{\Phi}{\Psi}:X\longrightarrow
[0,+\infty)$ attains the value $0$ at some point $\bar x$. This, because
the validity of $\dist{y}{\Psi(\bar x)}=0$ for every $y\in
\Phi(\bar x)$, as $\Psi(\bar x)$ is a closed set, implies
$\Phi(\bar x)\subseteq\Psi(\bar x)$.
The nonemptiness of the values taken by $\Psi$, along with the
boundedness of the values of $\Phi$, make the function $\excf{\Phi}{\Psi}$
real valued all over $X$.
So, take an arbitrary $x_0\in X$. In the case
$\excf{\Phi}{\Psi}(x_0)=0$, one has immediately $x_0\in\Inc{\Phi}{\Psi}
\ne\varnothing$ and the estimate in $(\ref{in:Incdist})$.
So, assume henceforth that $\excf{\Phi}{\Psi}(x_0)>0$.
Observe that, under the current hypotheses, the function $\excf{\Phi}{\Psi}$
turns out to be l.s.c. on $X$, by virtue of Lemma \ref{lem:lscexc}.
As it is obviously bounded from below and $X$ is complete,
then the Ekeland variational principle
applies. Accordingly, for every $\lambda>0$ there exists $x_\lambda
\in X$ such that
\begin{eqnarray*}
    \excf{\Phi}{\Psi}(x_\lambda)\le\excf{\Phi}{\Psi}(x_0),
\end{eqnarray*}
\begin{eqnarray}     \label{in:EVP2}
     d(x_\lambda,x_0)\le\lambda,
\end{eqnarray}
\begin{eqnarray}     \label{in:EVP3}
    \hskip1cm \excf{\Phi}{\Psi}(x_\lambda)<\excf{\Phi}{\Psi}(x)+
    {\excf{\Phi}{\Psi}(x_0)\over\lambda}d(x,x_\lambda),
    \quad\forall x\in X\backslash\{x_\lambda\}.
\end{eqnarray}
Take $\lambda=\excf{\Phi}{\Psi}(x_0)/(\alpha-\beta)$. The claim
to be proved is that $\excf{\Phi}{\Psi}(x_\lambda)=0$. Ab absurdo,
assume that $r_\lambda=\excf{\Phi}{\Psi}(x_\lambda)>0$. Since
$\Phi(x_\lambda)\subseteq\ball{\Psi(x_\lambda)}{r_\lambda}$, the set-covering
property of $\Psi$ enables one to state the existence of
$u\in\ball{x_\lambda}{r_\lambda/\alpha}$ such that
$$
    \Psi(u)\supseteq\ball{\Psi(x_\lambda)}{r_\lambda}\supseteq
    \Phi(x_\lambda).
$$
Therefore, recalling that the function $x\mapsto\exc{\Phi(x)}{\Psi(u)}$
is Lipschitz on $X$, with constant $\beta$, as a consequence
of the assumption on $\Phi$ and Remark \ref{rem:Lipexc},
one obtains
\begin{eqnarray}    \label{in:estexcatu}
   \excf{\Phi}{\Psi}(u)\le\exc{\Phi(x_\lambda)}{\Psi(u)}+\beta d(u,x_\lambda)
   =\beta d(u,x_\lambda).
\end{eqnarray}
Notice that it must be $u\in X\backslash\{x_\lambda\}$, otherwise
it would be $\excf{\Phi}{\Psi}(x_\lambda)=0$. By choosing $x=u$
in inequality $(\ref{in:EVP3})$ and taking into account inequality
$(\ref{in:estexcatu})$, one finds
$$
   \excf{\Phi}{\Psi}(x_\lambda)<\excf{\Phi}{\Psi}(u)+(\alpha-\beta)
   d(u,x_\lambda)\le\alpha d(u,x_\lambda)\le r_\lambda,
$$
which leads to an evident contradiction. This allows one to conclude
that  $\excf{\Phi}{\Psi}(x_\lambda)=0$, and hence $x_\lambda\in
\Inc{\Phi}{\Psi}\ne\varnothing$. From inequality $(\ref{in:EVP2})$,
it readily follows
\begin{eqnarray*}
   \dist{x_0}{\Inc{\Phi}{\Psi}}\le d(x_0,x_\lambda)
   \le{\exc{\Phi(x_0)}{\Psi(x_0)}\over\alpha-\beta}.
\end{eqnarray*}
By arbitrariness of $x_0$, this completes the proof.
\end{proof}

When $\Phi$ is a single-valued mapping, Theorem \ref{thm:incpoint}
allows one to obtain, as a special case, a well-known result
about the existence and error bound estimates for coincidence
points of the inclusion $\Phi(x)\in\Psi(x)$. Nevertheless, in order
to achieve such a result, the conventional notion of covering
is actually enough (see \cite{Arut07}). In contrast to this, as far as set-inclusion
points are concerned, the set-covering property
plays an essential role. The following counterexample
shows that such a property can not be replaced with
the usual covering notion for set-valued mappings.

\begin{example}
Consider the set-valued mapping $\Psi:\R\rightrightarrows\R^2$,
introduced in Example \ref{ex:covnotsetcov}, that is covering
with constant $\alpha=1$.
It is easy to check that this mapping is also Hausdorff u.s.c. on $\R$.
Define a further mapping $\Phi:\R\rightrightarrows\R^2$ as follows
$$
    \Phi(x)=\left({|x|\over 2}+1\right)\B.
$$
Since it is
$$
    \exc{\Phi(x_1)}{\Phi(x_2)}\le{1\over 2}|x_1-x_2|,\quad
    \forall x_1,\, x_2\in\R,
$$
$\Phi$ turns out to be Lipschitz on $\R$ with constant $\beta
=1/2<1=\alpha$ and bounded value. Nonetheless, in this case it happens that
$\Inc{\Phi}{\Psi}=\varnothing$, as one readily observes,
being $\inte\Psi(x)=\varnothing$ for every $x\in\R$.
\end{example}

A related application of the notion of set-covering concerns the
fixed point theory for multivalued mappings.

\begin{definition}
An element $x$ of a metric space $X$ is said to be a {\it strongly
fixed point} of a set-valued mapping $\Psi:X\rightrightarrows X$
if for some $r>0$ it is
$$
    \ball{x}{r}\subseteq\Psi(x).
$$
The set of all strongly fixed point of $\Psi$ is denoted henceforth
by $\SFix{\Psi}$.
\end{definition}

In the following proposition, strongly fixed points are shown to arise
in connection with set-covering mappings with constant greater
than 1 (a sort of expanding mappings).

\begin{proposition}
Let $\Psi:X\rightrightarrows X$ be a set-valued mapping defined on
a vector space, endowed with a complete and shift invariant metric.
If $\Psi$ is u.s.c. and set-covering on $X$,
with constant $\alpha>1$, then $\SFix{\Psi}\ne\varnothing$ and it holds
$$
    \dist{x}{\SFix{\Psi}}\le{\dist{x}{\Psi(x)}\over\alpha-1},
    \quad\forall x\in X.
$$
Moreover, $\SFix{\Psi}$ is  a dense subset of the set of all fixed points
of $\Psi$.
\end{proposition}

\begin{proof}
Fix arbitrary $x_0\in X$ and $r>0$ and consider the set-valued mapping
$\Phi_r:X\rightrightarrows X$ given by
$$
    \Phi_r(x)=\ball{x}{r}.
$$
Let us show that, under the current hypotheses,
$\Phi_r$ is Lipschitz with constant $\beta=1$. Notice indeed
that, by the shift invariance of the metric on $X$, one has $\ball{x}{r}=
x+\ball{\nullv}{r}$. Thus, taken $x_1,\, x_2\in X$, if $y_1\in\Phi_r(x_1)=
x_1+\ball{\nullv}{r}$, for some $u\in\ball{\nullv}{r}$ it results in
\begin{eqnarray*}
    \dist{y_1}{\Phi_r(x_2)} &=&  \dist{y_1}{x_2+\ball{\nullv}{r}}=
       \dist{y_1-x_2}{\ball{\nullv}{r}}  \\
    &=&  \dist{x_1+u-x_2}{\ball{\nullv}{r}}   \\
    &\le & d(x_1,x_2)+\dist{u}{\ball{\nullv}{r}}=d(x_1,x_2),
\end{eqnarray*}
whence
$$
    \exc{\Phi_r(x_1)}{\Phi_r(x_2)}\le d(x_1,x_2).
$$
Since $\alpha>1$, then it is possible to apply Theorem \ref{thm:incpoint},
according to which $\Inc{\Phi_r}{\Psi}\ne\varnothing$ and
$$
   \dist{x_0}{\Inc{\Phi_r}{\Psi}}\le {\exc{\Phi_r(x_0)}{\Psi(x_0)}
   \over\alpha-1}.
$$
Now, observe that $\SFix{\Psi}=\cup_{r>0}{\Inc{\Phi_r}{\Psi}}\ne
\varnothing$. By using again the shift invariance of the metric,
one obtains
\begin{eqnarray*}
   \dist{x_0}{\SFix{\Psi}} &\le& \inf_{r>0}{\exc{\Phi_r(x_0)}{\Psi(x_0)}
   \over\alpha-1} \\
   &=&\inf_{r>0}\sup_{u\in\ball{\nullv}{r}}{\dist{x_0+u}{\Psi(x_0)}\over\alpha-1} \\
   &\le & \inf_{r>0}\sup_{u\in\ball{\nullv}{r}}{d(x_0+u,x_0)+
   \dist{x_0}{\Psi(x_0)}\over\alpha-1}   \\
   &=& {\dist{x_0}{\Psi(x_0)}\over\alpha-1}.
\end{eqnarray*}
This proves the first assertion in thesis. The second one is a
straightforward consequence of the first.
\end{proof}

\subsection{Applications to exact penalization}

Of course, in force of its versatility,
a set-inclusion problem associated with a given pair of multivalued
mappings may appear among the constraints of optimization problems.
In the remaining part of this section, it is shown how the error
bound estimate provided in Theorem \ref{thm:incpoint} can be
exploited in deriving exact penalization results specific for problems
with a constraint of this type.
Exact penalization is a well-known approach for the treatment of
variously constrained optimization problems, whose effectiveness
is recognized from the theoretical as well as from the algorithmic
viewpoint. Essentially, it consists in reducing a given constrained
extremum problem to an unconstrained one, by replacing its objective
functional with a so-called penalty functional, which is obtained by
adding to the original objective functional a term properly quantifying
the constraint violation (see, for instance, \cite{Zasl10}).
Let us focus here on constrained optimization problems of the
form
$$
   \min\varphi(x) \quad\hbox{ subject to }\quad x\in
   {\mathcal R}=\Inc{\Phi}{\Psi},    \leqno({\mathcal P})
$$
where the objective functional $\varphi:X\longrightarrow\R\cup
\{\pm\infty\}$ and the multivalued mappings $\Psi:X
\rightrightarrows Y$ and  $\Phi:X\longrightarrow{\mathcal B}(Y)$
are given problem data. The penalty functional
$$
    \varphi_l(x)=\varphi(x)+l\cdot\excf{\Phi}{\Psi}(x)
$$
enables one to associate with problem $({\mathcal P})$ the unconstrained
problem
$$
   \min_{x\in X}\varphi_l(x).    \leqno({\mathcal P}_l)
$$
Conditions ensuring that a local solution to $({\mathcal P})$ is
also a local solution to $({\mathcal P}_l)$, provided that $l$ is
large enough, namely ensuring the existence of an exact penalty
functional,
turns out to be useful for formulating necessary
optimality conditions for problem $({\mathcal P})$. 
In the next result, a condition of this type is established.

\begin{theorem}
Let $\bar x\in {\mathcal R}$ be a local solution to $({\mathcal P})$.
Suppose that

\begin{enumerate}

\item[(i)] $\varphi$ is locally Lipschitz near $\bar x$, with constant
$l_\varphi>0$;

\item[(ii)] $(X,d)$ is metrically complete;

\item[(iii)] $\Psi$ is Hausdorff u.s.c. and set-covering on $X$,
with constant $\alpha>0$;

\item[(iv)] $\Phi$ is Lipschitz on $X$ with constant $\beta
\in [0,\alpha)$.

\end{enumerate}

\noindent Then, the penalty functional $\varphi_l$ is exact at $\bar x$ (i.e.
$\bar x$ is an unconstrained local minimizer
of $\varphi_l$), for every $l\ge{l_\varphi\over\alpha-\beta}$.

\end{theorem}

\begin{proof}
According to the hypothesis (i) there exists $r_\varphi>0$
such that
\begin{eqnarray}     \label{in:locLipvarphi}
   |\varphi(x_1)-\varphi(x_2)|\le l_\varphi d(x_1,x_2),\quad\forall
   x_1,\, x_2\in\ball{\bar x}{r_\varphi}.
\end{eqnarray}
Since $\bar x\in {\mathcal R}$ is a local solution to problem
$({\mathcal P})$, there exists $r_0>0$ such that
$$
   \varphi(x)\ge\varphi(\bar x),\quad\forall x\in\ball{\bar x}{r_0}
   \cap{\mathcal R}.
$$
Choose $\hat r>0$ in such a way that $\hat r<\min\{r_0/2,\, 
r_\varphi/2\}$. With this choice, let us show that for any
$l\ge{l_\varphi\over\alpha-\beta}$ it is true that
\begin{eqnarray}   \label{in:unconstminthesis}
   \varphi_l(\bar x)\le\varphi_l(x),\quad\forall x\in
   \ball{\bar x}{\hat r}.
\end{eqnarray}
In fact, this inequality trivially holds true if $x\in\ball{\bar x}{\hat r}
\cap{\mathcal R}$, so it remains to show its validity in the case
$x\in\ball{\bar x}{\hat r}\backslash{\mathcal R}$. For those $x$
for which it is $\excf{\Phi}{\Psi}(x)\ge{r_0\over 2}(\alpha-\beta)$, on
account of inequality $(\ref{in:locLipvarphi})$ it results in
\begin{eqnarray*}
  \varphi(x) &\ge& \varphi(\bar x)-l_\varphi d(x,\bar x)\ge
  \varphi(\bar x)-l_\varphi\hat r>\varphi(\bar x)-l_\varphi
  {r_0\over 2}  \\
 &\ge& \varphi(\bar x)-l_\varphi{\excf{\Phi}{\Psi}(x)\over
  \alpha-\beta}\ge\varphi(\bar x)-l\excf{\Phi}{\Psi}(x),
\end{eqnarray*}
which gives $\varphi_l(\bar x)\le\varphi_l(x)$.
On the other hand,
for those $x$ for which it is $\excf{\Phi}{\Psi}(x)<{r_0\over 2}
(\alpha-\beta)$, it is possible to find $\epsilon_0>0$ such that
${\excf{\Phi}{\Psi}\over\alpha-\beta}(1+\epsilon_0)\le{r_0\over 2}$.
So, take an arbitrary $\epsilon\in (0,\epsilon_0)$. Owing to
the error bound estimate in $(\ref{in:Incdist})$, which
can be employed under the hypotheses currently in force, one
deduces the existence of $x_\epsilon\in{\mathcal R}$ such that
\begin{eqnarray}    \label{in:projR}
   d(x,x_\epsilon)<{\excf{\Phi}{\Psi}(x)\over\alpha-\beta}(1+\epsilon).
\end{eqnarray}
Observe that
$$
   d(x_\epsilon,\bar x)\le d(x_\epsilon,x)+d(x,\bar x)<
  {\excf{\Phi}{\Psi}(x)\over\alpha-\beta}(1+\epsilon)+
  {r_0\over 2}<r_0.
$$
Therefore, it is $x_\epsilon\in\ball{\bar x}{r_0}\cap{\mathcal R}$,
what allows one to infer that $\varphi(x_\epsilon)\ge\varphi
(\bar x)$. By using again inequality  $(\ref{in:locLipvarphi})$,
this time together with $(\ref{in:projR})$, one finds
$$
   \varphi(x)\ge\varphi(x_\epsilon)-l_\varphi d(x,x_\epsilon)
   >\varphi(\bar x)-l_\varphi{\excf{\Phi}{\Psi}(x)\over
   \alpha-\beta}(1+\epsilon).
$$
By passing to the limit as $\epsilon\to 0^+$, one readily obtains
inequality $(\ref{in:unconstminthesis})$, which was to be proved.
\end{proof}

By exploiting again the above error bound estimate for set-inclusion points,
it is possible to establish a sort of converse of the last result, which
is valid for global solutions.

\begin{proposition}
Let problem $({\mathcal P})$ admit global solutions. Suppose that:

\begin{enumerate}

\item[(i)] $\varphi$ is Lipschitz on $X$ with constant $l_\varphi>0$;

\item[(ii)] $(X,d)$ is metrically complete;

\item[(iii)] $\Psi$ is Hausdorff u.s.c. and set-covering on $X$,
with constant $\alpha>0$;

\item[(iv)] $\Phi$ is Lipschitz on $X$ with constant $\beta
\in [0,\alpha)$.

\end{enumerate}

\noindent Fix $\epsilon>0$ and set $l_\epsilon={(1+\epsilon)l_\varphi\over
\alpha-\beta}$. If $\hat x$ is a strict global solution to problem
$({\mathcal P}_{l_\epsilon})$, then $\hat x$ globally solves also $({\mathcal P})$.
\end{proposition}

\begin{proof}
By virtue of the error bound estimate $(\ref{in:Incdist})$, corresponding
to any $\epsilon>0$ an element $x_\epsilon\in {\mathcal R}$ can be found
such that
$$
  d(x_\epsilon,\hat x)\le{1+\epsilon\over\alpha-\beta}\exc{\Phi(\hat x)}
  {\Psi(\hat x)}.
$$
Denote by $\bar x\in {\mathcal R}$ a global solution of $({\mathcal P})$.
As $\hat x$ solves problem $({\mathcal P}_{l_\epsilon})$, by the last
inequality and the Lipschitz continuity of $\varphi$, one obtains
\begin{eqnarray*}
    \varphi(\bar x) &=&\varphi_{l_\epsilon}(\bar x)\ge
    \varphi_{l_\epsilon}(\hat x)=\varphi(\hat x)+l_\epsilon\exc{\Phi(\hat x)}
  {\Psi(\hat x)}  \\
   &\ge & \varphi(x_\epsilon)-l_\varphi d(\hat x,x_\epsilon)+
        l_\epsilon\exc{\Phi(\hat x)}{\Psi(\hat x)} \ge\varphi
        (x_\epsilon)\ge\varphi(\bar x).
\end{eqnarray*}
The consequent fact that $\varphi_{l_\epsilon}(\hat x)=\varphi_{l_\epsilon}
(x_\epsilon)$, since $\hat x$ is strict as a global solution to
$({\mathcal P}_{l_\epsilon})$, entails that $\hat x=x_\epsilon$, so
that also $\hat x\in{\mathcal R}$. Thus, on account of the above inequalities,
it is possible to conclude that $\hat x$ is a global solution of
$({\mathcal P})$.
\end{proof}

Another approach to penalization methods in constrained optimization
rests upon the concept of problem calmness, which was introduced
by R.T. Rockafellar. This approach requires to regard a given problem as a
particular specialization of a class of parameterized problems.
In the case under study, the following class will be considered
$$
   \min\varphi(x) \quad\hbox{ subject to }\quad x\in
   {\mathcal R}(p)=\{x\in X:\ \Phi(p,x)\subseteq\Psi(p,x)\},
     \leqno({\mathcal P}_p)
$$
where the data $\Phi:P\times X\longrightarrow{\mathcal B}(Y)$
and $\Psi:P\times X\rightrightarrows Y$ now depend also on $p\in P$,
with $(P,d)$ denoting a metric space of parameters. Notice that,
unless suitable assumptions are introduced, one can not expect
in general that $\dom{\mathcal R}=P$. With respect
to this problem parameterization, the penalty functional
$\varphi_l:P\times X\longrightarrow\R\cup\{\pm\infty\}$ becomes
$$
   \varphi_l(p,x)=\varphi(x)+l\cdot\exc{\Phi(p,x)}{\Psi(p,x)}.
$$

\begin{definition}
Let $\bar p\in P$ and let $\bar x\in{\mathcal R}(\bar p)$ be a
local minimizer of the problem $({\mathcal P}_{\bar p})$.
Problem $({\mathcal P}_{\bar p})$ is called {\it calm} at $\bar p$
if there exist positive real constants $r$ and $\zeta$ such that
$$
   \varphi(x)\ge\varphi(\bar x)-\zeta d(p,\bar p),\quad\forall
   x\in\ball{\bar x}{r}\cap{\mathcal R}(\bar p),\ \forall
   p\in\ball{\bar p}{r}.
$$
\end{definition}

Appeared firstly in \cite{Clar76}, since then the above property
become a fundamental regularity condition pervading the
study of the sensitivity behaviour of variational problems,
in the presence of perturbations (see \cite{RocWet98}).
In the present context,
the introduction of problem calmness allows one to avoid
the assumption of Lipschitz continuity on the objective
functional. The price to be paid for enlarging the class of
problems, to which the penalization technique can be applied,
consists in a regularity requirement on the feasible region of
the problem class, formalized as follows.

\begin{definition}    \label{def:unisemireg}
A set-valued mapping $\Xi:P\rightrightarrows X$ between metric
spaces is said to be
{\it semiregular} at $\bar p\in P$, {\it uniformly over} $\Xi(\bar p)$, if
if there exist positive real constants $r$ and $\kappa$ such that
\begin{eqnarray}     \label{in:unisemireg}
    \dist{\bar p}{\Xi^{-1}(x)}\le\kappa d(x,\bar x),\quad\forall
   x\in\ball{\bar x}{r},\ \forall \bar x\in\Xi(\bar p).
\end{eqnarray}
\end{definition}

The property formulated in Definition \ref{def:unisemireg}
is an enhanced version of a regularity notion that, to the best
of the author's knowledge, was introduced in \cite{Krug09}.
The latter is known to correspond to the well-known Lipschitz lower
semicontinuity property for the inverse mapping of $\Xi$.

\begin{remark}    \label{rem:semiregreform}
It is readily seen that, whenever $\Xi:P\rightrightarrows X$ is
semiregular at $\bar p$, uniformly over $\Xi(\bar p)$, then it holds
\begin{eqnarray}    \label{in:unisemiregset}
    \dist{\bar p}{\Xi^{-1}(x)}\le\kappa \dist{x}{\Xi(\bar p)},\quad\forall
   x\in\ball{\Xi(\bar p)}{r/2}.
\end{eqnarray}
Indeed, taken $x\in\ball{\Xi(\bar p)}{r/2}\backslash\Xi(\bar p)$ and
any $\epsilon\in (0,1)$, there exists $\bar x_\epsilon\in\Xi(\bar p)$
such that
$$
   d(x,\bar x_\epsilon)<(1+\epsilon)\dist{x}{\Xi(\bar p)}<r.
$$
Then, inequality $(\ref{in:unisemireg})$ applies, so it results in
$$
    \dist{\bar p}{\Xi^{-1}(x)}\le\kappa d(x,\bar x_\epsilon)<
   \kappa(1+\epsilon)\dist{x}{\Xi(\bar p)},
$$
whence inequality $(\ref{in:unisemiregset})$ follows by arbitrariness
of $\epsilon\in (0,1)$. Actually, the validity of $(\ref{in:unisemiregset})$
is an equivalent reformulation of the uniform semiregularity of $\Xi$
at $\bar p$, as one checks immediately.
\end{remark}

One is now in a position to establish the next result about exact
penalization.

\begin{theorem}     \label{thm:calmexpen}
With reference to a parameterized family of problems $({\mathcal P}_p)$,
let $\bar x\in{\mathcal R}(\bar p)$ be a local minimizer of problem
$({\mathcal P}_{\bar p})$. Suppose that:

\begin{enumerate}

\item[(i)] $(X,d)$ is metrically complete;

\item[(ii)] $\varphi$ is l.s.c. at $\bar x$;

\item[(iii)] ${\mathcal R}:P\rightrightarrows X$ is semiregular
at $\bar p$, uniformly over ${\mathcal R}(\bar p)$;

\item[(iv)] problem $({\mathcal P}_{\bar p})$ is calm at $\bar p$;

\end{enumerate}

\noindent and there exists $\delta>0$ such that:

\begin{enumerate}

\item[(v)] $\Psi(p,\cdot):X\rightrightarrows Y$ is Hausdorff u.s.c.
and set-covering on $X$ with constant $\alpha_p>0$, for each
$p\in\ball{\bar p}{\delta}$;

\item[(vi)] $\Phi(p,\cdot):X\longrightarrow {\mathcal B}(Y)$ is
Lipschitz on $X$ with constant $\beta_p\in (0,\alpha_p)$, for each
$p\in\ball{\bar p}{\delta}$.

\end{enumerate}

\noindent Then, there exists $l>0$ such that the penalty functional
$\varphi_l(\bar p,\cdot)$ is exact at $\bar x$.
\end{theorem}

\begin{proof}
Let us start with noting that, under the current assumptions, by
virtue of Theorem \ref{thm:incpoint} it is $\dom{\mathcal R}\supseteq
\ball{\bar p}{\delta}$ and the following estimate holds
\begin{eqnarray}    \label{in:erbofeasparam}
  \qquad\qquad\dist{x}{{\mathcal R}(p)}\le{\exc{\Phi(p,x)}{\Psi(p,x)}
   \over\alpha_p-\beta_p},
  \quad\forall x\in\ X,\ \forall p\in\ball{\bar p}{\delta}.
\end{eqnarray}
Recall that, according to what has been noticed in Remark \ref{rem:Incclosed},
the mapping ${\mathcal R}$ is closed valued. By hypothesis (iii), in the
light of Remark \ref{rem:semiregreform}, there exist $r>0$ and
$\kappa>0$ such that
$$
    \dist{\bar p}{{\mathcal R}^{-1}(x)}\le\kappa
   \dist{x}{{\mathcal R}(\bar p)},\quad\forall x\in\ball{{\mathcal R}(\bar p)}
   {r}.
$$
From the last inequality, on the account of the estimate
$(\ref{in:erbofeasparam})$, it follows
\begin{eqnarray}    \label{in:erbopenal}
   \dist{\bar p}{{\mathcal R}^{-1}(x)}\le
   {\kappa\over\alpha_{\bar p}-\beta_{\bar p}}
   \exc{\Phi(\bar p,x)}{\Psi(\bar p,x)}, \\ \nonumber
    \forall x\in\ball{{\mathcal R}(\bar p)}{r}.
\end{eqnarray}
Assume now, ab absurdo, that for each $l>0$ the penalty functional
$\varphi_l(\bar p,\cdot)$ fails to be exact at $\bar x$, namely for each $l>0$
there exists $n\in\N$, with $n>l$, and $x_n\in\ball{\bar x}{1/n}$
such that
\begin{eqnarray}    \label{in:absurd}
     \varphi(x_n)+n\cdot\exc{\Phi(\bar p,x_n)}{\Psi(\bar p,x_n)}
    <\varphi(\bar x).
\end{eqnarray}
Since $\bar x$ is a local solution to problem $({\mathcal P}_{\bar p})$,
for each $n\in\N$ larger than a proper natural number it must be
$x_n\not\in{\mathcal R}(\bar p)$, that is, as multifunctions take
closed values, $\exc{\Phi(\bar p,x_n)}
{\Psi(\bar p,x_n)}>0$. Moreover, by virtue of the lower semicontinuity
of $\varphi$ at $\bar x$ and of the fact that the sequence
$(x_n)_{n\in\N}$ converges to $\bar x$ as $n\to\infty$, from inequality
$(\ref{in:absurd})$ one obtains
\begin{eqnarray*} 
    \limsup_{n\to\infty}n\cdot\exc{\Phi(\bar p,x_n)}{\Psi(\bar p,x_n)} &\le&
    \limsup_{n\to\infty}\ [\varphi(\bar x)-\varphi(x_n)]  \\
    &=&\varphi(\bar x)-\liminf_{n\to\infty}\varphi(x_n)\le 0.
\end{eqnarray*} 
Then one deduces that
\begin{eqnarray}    \label{eq:converxn}
  \lim_{n\to\infty}\ \exc{\Phi(\bar p,x_n)}{\Psi(\bar p,x_n)}=0.
\end{eqnarray}
Since, as already observed, $x_n\longrightarrow\bar x$ as $n\to\infty$,
it is possible to assume without loss of generality that $x_n\in
\ball{\bar x}{r}$, and hence $x_n\in\ball{{\mathcal R}(\bar p)}{r}$.
This fact enables one to apply inequality $(\ref{in:erbopenal})$,
from which one obtains
$$
   \dist{\bar p}{{\mathcal R}^{-1}(x_n)}\le
   {\kappa\over\alpha_{\bar p}-\beta_{\bar p}}
   \exc{\Phi(\bar p,x_n)}{\Psi(\bar p,x_n)}.
$$
This means that, corresponding to a costant $\tilde\kappa>\kappa$,
a suitable $p_n\in {\mathcal R}^{-1}(x_n)$ can be found, such that
the inequality
\begin{eqnarray}    \label{in:pestimerbo}
    d(p_n,\bar p)<{\tilde\kappa\over\alpha_{\bar p}-\beta_{\bar p}}
   \exc{\Phi(\bar p,x_n)}{\Psi(\bar p,x_n)}
\end{eqnarray}
holds for every $n\in\N$ large enough. Observe that, as $x_n\in
{\mathcal R}(p_n)$ and $x_n\not\in{\mathcal R}(\bar p)$, then it
must be $p_n\ne\bar p$. By combining inequalities $(\ref{in:pestimerbo})$
and $(\ref{in:absurd})$, one finds
$$
    {\tilde\kappa\over\alpha_{\bar p}-\beta_{\bar p}}\cdot
    {\varphi(x_n)-\varphi(\bar x)\over d(p_n,\bar p)}\le
    {\varphi(x_n)-\varphi(\bar x)\over\exc{\Phi(\bar p,x_n)}{\Psi(\bar p,x_n)}}
    <-n,
$$
wherefrom, for $x_n\in\ball{\bar x}{r}\cap{\mathcal R}(p_n)$ and
every $n\in\N$ large enough, one gets
\begin{eqnarray}   \label{in:contradict}
   \varphi(x_n)<\varphi(\bar x)- {n(\alpha_{\bar p}-\beta_{\bar p})
   \over\tilde\kappa}d(p_n,\bar p).
\end{eqnarray}
Notice that, owing to inequality $(\ref{eq:converxn})$, the estimate
in $(\ref{in:pestimerbo})$ entails that $p_n\longrightarrow\bar p$ as
$n\to\infty$. Consequently, inequality $(\ref{in:contradict})$ contradicts
the hypothesis (iv) about the calmness at $\bar x$ of problem
$({\mathcal P}_{\bar p})$. Thus, the proof is complete.
\end{proof}

\begin{remark}
Among the hypotheses of Theorem \ref{thm:calmexpen}, (iii) and (iv)
are not directly formulated in terms of problem data. Conditions
for mapping ${\mathcal R}$ to be uniformly regular at $\bar p$
can be derived by working, under proper assumptions on $\Phi$ and $\Psi$,
the general characterization for the semiregularity of a mapping
$\Xi:P\rightrightarrows X$, which is expressed by the positivity
of the constant
$$
     \vartheta[\Xi](\bar p,\bar x)=\liminf_{x\to\bar x}
     {\dist{x}{\Xi(\bar p)}\over\dist{\bar p}{\Xi^{-1}(x)}}
$$
(see \cite{Krug09}).
A sufficient condition for a parameterized problem $({\mathcal P}_{\bar p})$
to be calm can be expressed in terms of calmness from below of the
related value function $\nu:P\longrightarrow\R\cup\{\pm\infty\}$,
defined as
$$
    \nu(p)=\inf_{x\in{\mathcal R}(p)}\varphi(x).
$$
More precisely, if $\bar x\in{\mathcal R}(\bar p)$ is a global
solution to problem $({\mathcal P}_{\bar p})$, then $({\mathcal P}_{\bar p})$
is calm $\bar x$ povided that $\nu$ is calm from below at $\bar p$, i.e.
$$
    \liminf_{p\to\bar p}{\nu(p)-\nu(\bar p)\over d(p,\bar p)}>-\infty
$$
(see, for more details, \cite{Uder10}).
\end{remark}

Theorem \ref{thm:calmexpen} provides a sufficient condition for the
exactness of the penalty functional, where problem calmness plays
a crucial role. In order to enlighten the intriguing connection between
these two properties, this subsection is concluded by a proposition
that, in the setting under examination, singles out certain conditions
upon which from exact penalization it is possible to derive
problem calmness.

\begin{proposition}
With reference to a parameterized family of problems $({\mathcal P}_p)$,
let $\bar x\in {\mathcal R}(\bar p)$ be a local minimizer of
$({\mathcal P}_{\bar p})$. Suppose that:

\begin{enumerate}

\item[(i)] mapping $\Phi:P\times X\longrightarrow{\mathcal B}(Y)$
is locally Lipschitz around $(\bar p,\bar x)$;

\item[(ii)] mapping $\Psi$ is partially Lipschitz u.s.c. at $\bar p$,
uniformly in $x$, i.e. there exist $r>0$ and $\zeta>0$ such that
$$
    \Psi(p,x)\subseteq\ball{\Psi(\bar p,x)}{\zeta d(p,\bar p)},
    \quad\forall p\in\ball{\bar p}{\zeta},\ \forall x\in\ball{\bar x}{\zeta};
$$

\item[(iii)] there exists $l>0$ such that $\varphi_l(\bar p,\cdot)$
is an exact penalty functional.

\end{enumerate}

\noindent Then, problem $({\mathcal P}_{\bar p})$ is calm
at $\bar x$.
\end{proposition}

\begin{proof}
The thesis can be proved again by a reductio ad absurdum.
So assume $({\mathcal P}_{\bar p})$ to be not calm at $\bar x$. This
amounts to say that  for each $n\in\N$ it is possible to find $p_n
\in\ball{\bar p}{1/n}\backslash\{\bar p\}$ and $x_n\in {\mathcal R}
(p_n)\cap\ball{\bar x}{1/n}$ such that
\begin{eqnarray}        \label{in:notcalmabsurd}
     \varphi(x_n)<\varphi(\bar x)-n d(p_n,\bar p).
\end{eqnarray}
By hypothesis (ii), since $p_n\to\bar p$ and $x_n\to\bar x$ as
$n\to\infty$, one has
\begin{eqnarray}      \label{in:Lusc}
    \exc{\Psi(p_n,x_n)}{\Psi(\bar p,x_n)}\le\zeta d(p_n,\bar p),
\end{eqnarray}
for each $n\in\N$ large enough.
On the other hand, since $\Phi$ is locally Lipschitz around
$(\bar p,\bar x)$, for some $\tau>0$ it is true that
\begin{eqnarray}    \label{in:locLip}
     \exc{\Phi(\bar p,x_n)}{\Phi(p_n,x_n)}\le\tau d(p_n,\bar p),
\end{eqnarray}
for every $n\in\N$ large enough. By recalling that $x_n\in
{\mathcal R}(p_n)$, so that $\Phi(p_n,x_n)\subseteq\Psi(p_n,x_n)$,
from inequalities $(\ref{in:locLip})$ and $(\ref{in:Lusc})$ one
obtains
\begin{eqnarray*}
    \exc{\Phi(\bar p,x_n)}{\Psi(\bar p,x_n)} &\le&  \exc{\Phi(\bar p,x_n)}{\Phi(p_n,x_n)} \\
        &+ &\exc{\Phi(p_n,x_n)}{\Psi(p_n,x_n)}  \\
        &+ &\exc{\Psi(p_n,x_n)}{\Psi(\bar p,x_n)} \\
        &\le & (\tau+\zeta)d(p_n,\bar p).
\end{eqnarray*}
By combining the above estimate with inequality $(\ref{in:notcalmabsurd})$,
one finds
$$
    \varphi(x_n)<\varphi(\bar x)-{n\over\tau+\zeta}\exc{\Phi(\bar p,x_n)}{\Psi(\bar p,x_n)}.
$$
As the last inequality is true for each $n\in\N$ larger than a
proper natural and for the corresponding $x_n\in\ball{\bar x}{1/n}$,
the hypothesis (iii)
about the existence of an exact penalty functional turns out to
be contradicted. Thus, the argument by contradiction is complete.
\end{proof}


\end{document}